\documentclass{amsart}
\newtheorem{theorem}{Theorem}[section]

\theoremstyle{definition}
\newtheorem{definition}[theorem]{Definition}
\newtheorem{example}[theorem]{Example}

\theoremstyle{remark}

\numberwithin{equation}{section}
\newcommand{\norm}[1]{\left\lVert#1\right\rVert}

\newcommand{\floor}[1]{\left\lfloor#1\right\rfloor}

\usepackage{amsfonts}
\usepackage{color,graphicx,amssymb}
\usepackage{graphicx}
\begin{document}
\title{On the regularity and approximation of invariant densities for random continued fractions}
\author{Toby Taylor-Crush}
    %Address of record for the research reported here
    \address{Department of Mathematical Sciences, Loughborough University, 
Loughborough, Leicestershire, LE11 3TU, UK}
\email{T.Taylor-Crush@lboro.ac.uk}
\thanks{We would like to thank anonymous referee for useful comments and suggestions that improved the content and presentation of the paper.}
\begin{abstract}
We study perturbations of random dynamical systems whose associated transfer operators admit a uniform spectral gap. We provide a $k^{\text{th}}$-order approximation for the invariant density of the associated random dynamical system. We apply our result to random continued fractions. 
\end{abstract}
\maketitle
\pagestyle{myheadings} 
\markboth{Invariant densities for random continued fractions }{Toby Taylor-Crush}
\section{Introduction}
When considering the long term statistics of a dynamical system, the associated invariant densities are indispensable. A natural question for the study of dynamics is to ask how the invariant densities change with some perturbation. It is known that for certain classes of maps the change in the density is smooth with certain perturbations, this is called linear response and allows us to find a first order approximation of invariant density for perturbed system. See \cite{Baladi} for a survey on linear response for deterministic systems and \cite{GS} for a recent work in the setting of random dynamical systems.  In this note, we work in the setting of random dynamical systems and work with a particular case of the systems studied in  \cite{Wael}: the case when the transfer operators admit a uniform spectral gap on a suitable Banach space. Under higher regularity assumptions on the constituent maps of the random system, we provide a higher order approximation of the invariant density of the underlying perturbed random system. We then apply this result to provide a higher order approximation of the invariant density of the Gauss-R\'enyi random system, which is the main motivation of our work. Such a system was first introduced in \cite{Kalle} where the authors describe the map in detail and demonstrate the existence of an invariant density for the system, called $h_p$. They then go on to describe various properties, such as it being absolutely continuous with respect to the Lebesgue measure, bounded away from $0$ and that the map is mixing with respect to the invariant measure. However, unlike in the classical Gauss map case, the formula of the invariant density of such a random system is unknown. Hence, it is natural to ask for an approximation of such an invariant density.

The invariant densities of various similar numerical expansions, like continued fraction expansions, have been widely studied in the past. In \cite{MeasureForMap} explicit formulae in the form of integrals are given for the invariant density of a class of continued fraction type maps. In \cite{Brjuno} they find the invariant densities of a one parameter family of continued fraction expansions while studying Brjuno functions. In \cite{DdeV} the existence and uniqueness of the invariant measure of the generating function of random $\beta$-expansions of numbers is established and shown to be mutually singular with the measure of maximal entropy. In \cite{NonRegular} they find the invariant density for non-regular continued fractions by proving a Gauss-Kuzmin type theorem. In \cite{ContFracDensitys} they find many invariant densities of maps for continued fractions with finitely many digits. They use various methods including natural extensions to find exact invariant densities and found approximations using the Gauss-Kuzmin-L\'evy theorem and by statistical approaches. In \cite{DO} they are able to demonstrate the existence of invariant probability measures for the maps relating to random $N$-continued fractions, which they study in a similar way to \cite{Kalle}. In \cite{Wael} they provide a first order approximation of the invariant density of the random continued fractions studied in \cite{Kalle}. In this note we show that an approximation of any order $k$ can be obtained for such a density. 

The layout of this paper is as follows, section \ref{Preliminaries} we specify the type of system we are dealing with, we describe a problem which can be solved using invariant densities, and a very similar problem directly relating to the subject of this paper. In section \ref{Sect:Result} we prove the main theorem of this paper, and in section \ref{Sect:Example} we discuss the application of this result to the problem described in section \ref{Preliminaries}. In particular, we provide an approximation to the distribution of the $n^{th}$ digit of a random continued fraction expansion, in general and we work out the computations of a particular example (see subsection \ref{subsect:Example}).
\section{Preliminaries}\label{Preliminaries}
Consider a family of maps $T_\omega:[0,1]\to[0,1]$ non-singular with respect to the Lebesgue measure $m$, with transfer operator $L_{T_\omega}$. Given some density $f$ on $[0,1]$, $L_{T_\omega}f$ is the density after $T_\omega$ has been applied to $[0,1]$ with density $f$. Let $(\Omega,\mathcal{F},\mathbb{P})$ be a probability space. We study independent identically distributed, with respect to $\mathbb{P}$, compositions of $T_\omega$. We call these compositions random maps. 

\subsection{Transfer operator for random maps}
For these systems we will need to describe their transfer operator.

The random map $(\Omega,\{T_\omega\},\mathbb{P})$ is understood as a Markov process with transition function
\begin{equation}
p(x,S)=\int_\Omega  1_S(T_\omega(x)) d\mathbb{P}(\omega)\nonumber
\end{equation}
with $x\in[0,1]$ and $S\subset [0,1]$. A measure $\mu$ is stationary if
\begin{equation}
\int_Xp(x,S)d\mu(x)=\mu(S).\label{eq:stat}
\end{equation}

For any $\phi\in L^\infty(X)$ and $\Phi\in L^1(X)$ we have
\begin{align}
\int_X\int_\Omega \phi\circ T_\omega\Phi d\mathbb{P}(\omega)d\mu(x)=&\int_\Omega  \int_X  \phi\circ T_\omega\Phi d\mu(x) d\mathbb{P}(\omega)\nonumber\\
=&\int_\Omega \int_X  \phi L_{T_{\omega}}\Phi d\mu(x)d\mathbb{P}(\omega)  \nonumber\\
=&\int_X\phi\int_\Omega   L_{T_{\omega}}\Phi d\mathbb{P}(\omega) d\mu(x)\label{eq:duality}
\end{align}
where
\begin{equation}
L_{T_\omega}\Phi:=\sum_{a\in A} |V'_{a,\omega}|\cdot\Phi \circ V_{a,\omega}
\end{equation}
is the transfer operator for $T_\omega$, with $V_{a,\omega}$ being the inverse of the map $T_\omega$ restricted to the interval labeled $a$, where the set of labels is $A$.
\begin{definition}
The Annealed transfer operator (sometimes called averaged transfer operator) is defined as
\begin{equation}
L_0\Phi:=\int_{\Omega} L_{T_\omega}\Phi d\mathbb{P}(\omega).
\end{equation}
We see in \eqref{eq:stat} that for any stationary measure $\mu$ absolutely continuous with respect to $m$ , with density $h$, has the property
\begin{equation}
L_0h=h.\nonumber
\end{equation}
The annealed transfer operator is what we will be calling the transfer operator of the random map $(\Omega,\{T_\omega\},\mathbb{P})$.
\end{definition}

\section{Higher order approximations of invariant densities for random maps}\label{Sect:Result}

Here we look at how we can estimate the absolutely continuous stationary measure of a system like the one described in section \ref{sect:randcontfrac}. The result below is quite general and can be applied to other random maps.

\subsection{Perturbed random maps}

Consider the interval $[0,1]$ and the Lebesgue measure $m$. If we have a system $(\Omega,\{T_\omega\},\mathbb{P}_0)$ with $\{T_\omega\}$ a family of non-singular maps, $T_\omega:[0,1]\to [0,1]$, such that for each $\omega\in\Omega$ there exists at most countably many branches of $T_\omega$ in intervals labeled $\underline{a}\in A$. Let $\{I_{\underline{a},\omega} | a\in A\}$ be a partition of $[0,1]$  half-open intervals of such that the restriction of $T_\omega$ to $I_{\underline{a},\omega}$ is surjective and $C^{l+1}$.

We want to study the perturbed system $(\Omega,\{T_\omega\},\mathbb{P}_\epsilon)$ where the map $T_{\omega}$ is chosen according to a new distribution $\mathbb{P}_\epsilon$. $L_\epsilon:C^{l}([0,1])\to C^{l}([0,1])$, is the annealed transfer operator under $\mathbb{P}_\epsilon$.
\begin{definition}
We say the family of operators $L_\epsilon$ has a uniform spectral gap if there exists a $\delta>0$, independent of $\epsilon$, such that for all $\epsilon$ small enough, $1$ is the unique eigenvalue of $L_\epsilon$ of modulus one, and the rest of the spectrum is contained in the ball centered at $0$ of radius $\delta$; i.e, $\sigma(L_\epsilon)\setminus\{1\}\subset B(0,\delta)$.
\end{definition}
\begin{theorem}\label{MainProp}
Assume that
\begin{enumerate}
	\item \label {Assumption1}$L_\epsilon$ has a uniform spectral gap in $C^{l}([0,1])$, $l\ge 1$.
	\item \label {Assumption2}$\epsilon \mapsto L_\epsilon h_0$ is $k$ times differentiable\footnote{In applications to specific situations, one may need to specify the relation between $k$ and $l$, see \cite{Wael} for example where $\partial_\epsilon^iL_\epsilon h_0$ involves a derivative in $x$. However, in other specific situations, such a restriction is not necessary. Indeed, when $L_\epsilon f=(1-\epsilon) L_0 f +\epsilon L_1 f$ such a restriction is not needed. 
%Note that this differentiability in $\epsilon$ is different from the differentiability in $x$, so we do not necessarily need $k\leq l$, however there are some situations, such as in \cite{Wael}, where $T_\omega$ is dependent on $\epsilon$, where such a restriction is required.
}
 as an element of $C^l([0,1])$ at $\epsilon=0$.
	\item \label {Assumption3}$\epsilon \mapsto (I-L_\epsilon)^{-1}G_i$ is $k-i$ times differentiable in $C^l([0,1])$ at $\epsilon=0$,
\end{enumerate}
 where $h_0$ is the invariant density of the system $(\Omega,\{T_\omega\},\mathbb{P}_0)$ and $G_i:=\partial_{\epsilon}^i L_{\epsilon} h_0|_{\epsilon=0}$ with $i=0,\dots,k-1$. Then $\epsilon\mapsto h_\epsilon$ is $k$ differentiable as an element of $C^l([0,1])$ at $\epsilon=0$.

In particular the invariant density of the system $(\Omega,\{T_\omega\},\mathbb{P}_\epsilon)$, $h_\epsilon$, can be approximated using the formula
\begin{equation}
h_\epsilon=h_0+\sum_{n=1}^k\frac{\epsilon^n}{n!} \partial_\epsilon^n h_\epsilon|_{\epsilon=0}+o(\epsilon^k)\nonumber
\end{equation}
where 
\begin{equation}\label{Result}
\partial_\epsilon^n h_\epsilon|_{\epsilon=0}=\sum_{i=1}^n \binom{n}{ i}H_{i,n-i}
\end{equation}
and
\begin{equation}H_{i,n-i}=\partial_\epsilon^{n-i}(I-L_\epsilon)^{-1}G_i|_{\epsilon=0}.\nonumber\end{equation}
\end{theorem}
\begin{proof}
This proof develops that of \cite{Wael} from the case of first order differentiation to the $k^{th}$ order. By assumption \ref{Assumption1} we have a uniform spectral gap in $\epsilon$, so we have that $(I-L_\epsilon)^{-1}$ is well defined on $C^l_0([0,1])$, the subset of $C^l([0,1])$ of elements with $0$ average, and is uniformly bounded in $\epsilon$. We have
\begin{equation}\label{Identheps}
h_\epsilon=(I-L_\epsilon)^{-1}(L_\epsilon-L_0)h_0+h_0.
\end{equation}
By assumption \ref{Assumption2} we have
\begin{equation}\label{eq:Proof1}
L_\epsilon h_0=\sum_{i=0}^k\frac{\epsilon^i}{i!}G_i+o(\epsilon^k).
\end{equation}
Notice that the first term in the sum on the right hand side of \eqref{eq:Proof1} is $L_0 h_0$, consequently
\begin{equation}
(L_\epsilon-L_0)h_0=\sum_{i=1}^k\frac{\epsilon^i}{i!} G_i+o(\epsilon^k).\nonumber
\end{equation}
Noticing that $G_i$ is zero-average, we may substitute this into \eqref{Identheps} to give
\begin{align}
h_\epsilon=&h_0+\sum_{i=1}^k \frac{\epsilon^i}{i!} (I-L_\epsilon)^{-1} G_i+ (I-L_\epsilon)^{-1} o(\epsilon^k)\nonumber\\ 
=&h_0+\sum_{i=1}^k \frac{\epsilon^i}{i!} (I-L_\epsilon)^{-1} G_i+o(\epsilon^k)\label{eq:heps},
\end{align}
where in the last step we have used the uniform boundedness, in $\epsilon$, of $(I-L_\epsilon)^{-1}$ on  $C^l_0([0,1])$. Now we take the Taylor series of each $(I-L_\epsilon)^{-1} G_i$, where we use assumption \ref{Assumption3},
\begin{equation}
(I-L_\epsilon)^{-1} G_i=\sum_{j=0}^{k-i}\frac{\epsilon^{j}}{j!}H_{i,j}+o(\epsilon^{k-i})\nonumber
\end{equation}
 which we can substitute into \eqref{eq:heps} to get
\begin{align}
h_\epsilon&=h_0+\sum_{i=1}^k \frac{\epsilon^i}{i!}\left[\sum_{j=0}^{k-i}\frac{\epsilon^j}{j!} H_{i,j}+o(\epsilon^{k-i})\right]+o(\epsilon^k)\nonumber\\
&=h_0+\sum_{i=1}^k \left[\sum_{j=0}^{k-i}\frac{\epsilon^{i+j}}{i!j!}H_{i,j}+\epsilon^i o(\epsilon^{k-i})\right]+o(\epsilon^k)\nonumber\\
&=h_0+\sum_{i=1}^k \left[\sum_{j=0}^{k-i}\frac{\epsilon^{i+j}}{i!j!} H_{i,j}\right]+o(\epsilon^k)\nonumber\\
&=h_0+\sum_{n=1}^k \sum_{i=1}^n\frac{\epsilon^{i+(n-i)}}{i!(n-i)!} H_{i,n-i}+o(\epsilon^k)\nonumber\\
&=h_0+\sum_{n=1}^k\epsilon^n\sum_{i=1}^n\frac{1}{i!(n-i)!}H_{i,n-i}+o(\epsilon^k)\nonumber\\
&=h_0+\sum_{n=1}^k\frac{\epsilon^n}{n!}\sum_{i=1}^n\frac{n!}{i!(n-i)!}H_{i,n-i}+o(\epsilon^k)\nonumber\\
&=h_0+\sum_{n=1}^k\frac{\epsilon^n}{n!}\sum_{i=1}^n \binom{n}{ i}  H_{i,n-i}+o(\epsilon^k).\nonumber
\end{align}
This finishes the proof.
\end{proof}

\section{The Gauss-R\'enyi random map}\label{Sect:Example}

\subsection{Random continued fractions}\label{sect:randcontfrac}
Consider the continued fraction representation of some $x\in[0,1]\backslash \mathbb{Q}$
\begin{equation}
x=\dfrac{1}{a_1+\dfrac{1}{a_2+\dfrac{1}{a_3+\dots}}}.\label{eq:ContinuedFraction}
\end{equation}
We may write $[x]=[a_1,a_2,a_3,\dots]$ to represent this continued fraction. A question posed by Gauss (who presented a solution in a letter to Laplace) in 1800, with proofs and convergence shown by Kuzmin \cite{Kuzmin} and improved by L\'evy \cite{Levy}, was to find the distribution of numbers $x\in[0,1]$ such that $a_n=N$ for some $N\in\mathbb{N}$ and $n>>1$. Gauss was able to show that
\begin{equation}\label{eq:GaussFormula}
\lim_{n\to\infty}m(\{x\in[0,1]:a_n=N\})=\frac{1}{\log{2}}\log{\frac{1+\frac{1}{N}}{1+\frac{1}{N+1}}}
\end{equation}
where $m$ is the Lebesgue measure.

We would like to answer the same question for a different system, the system of randomly choosing semi-regular continued fractions for a given $x$ as discussed in \cite{Kalle}.
\begin{definition}
A semi-regular continued fraction is a representation of an $x\in[-1,1]\backslash\mathbb{Q}$ that takes the form,
\begin{equation}
x=\dfrac{(-1)^{\omega_0}}{a_1+\dfrac{(-1)^{\omega_1}}{a_2+\dots}}\nonumber
\end{equation}
where $\omega_n\in\{0,1\}$ and $a_n\in\mathbb{N}$.
\end{definition}
 We choose these from $\omega=[\omega_1,\omega_2,\omega_3,\dots]\in\Omega^{\mathbb{N}}=\{0,1\}^{\mathbb{N}}$ according to some distribution $\mathbb{P}$. It should be noted that $\omega_0$ depends on whether $x\in(0,1]$ or $x\in[-1,0)$ and so is not chosen randomly with the others. We should note that while there is a unique continued fraction representation of $x\in[0,1]$ there are uncountably many semi-regular continued fractions representations of $x\in[-1,1]$.

To study random continued fractions we need a new dynamical system $K:\Omega^{\mathbb{N}}\times[-1,1]\to\Omega^{\mathbb{N}}\times[-1,1]$ that serves the same function as the Gauss map $T_0$ did for the regular continued fractions problem. We will define this as
\begin{equation}
K(\omega,x)=(\sigma(\omega),K_2(\omega,x)).\nonumber
\end{equation}
where
\begin{equation}
K_2(\omega,x)=\frac{1}{|x|}-k-\omega_1\nonumber
\end{equation}
for $k=\floor{\frac{1}{|x|}}$, and $\sigma$ the left shift map. We have that $\sigma([x])=[K_2(\omega,x)]$, where $[x]$ is the list of letters of the continued fraction expansion of $x$. We see this by rearranging $K_2$ and writing
\begin{equation}
x=\frac{(-1)^{\omega_0}}{k+\omega_1+K_2(\omega,x)}.\nonumber
\end{equation}
We retrieve the digits of the semi-regular continued fraction using $d(\omega,x)=k+\omega_1$ and setting
\begin{equation}
d_n=d(K_2^{n-1}(\omega,x)).\nonumber
\end{equation}
This system can be written in terms of the Gauss map, $T_0$, and the R\'enyi map, $T_1$:
\begin{align}
T_0(x)=&\frac{1}{x}-\floor{\frac{1}{x}}\nonumber;\\
T_1(x)=&\frac{1}{1-x}-\floor{\frac{1}{1-x}},\nonumber
\end{align}
with $T_0(0)=0$ and $T_1(1)=0$. We can write $K_2$ as
\begin{equation}
K_2(\omega,x)=T_{\omega_0}(x+\omega_0)-\omega_1.\nonumber
\end{equation}

The properties of this system are not so easy to study, so we will look at another system which is conjugate to $K$, as was done in \cite{Kalle} to demonstrate the existence of a unique invariant probability measure for $K$. This system is the Gauss-R\'enyi map, $R:\Omega^{\mathbb{N}} \times [0,1]\to\Omega^{\mathbb{N}} \times [0,1]$
\begin{equation}
R(\omega,x)=(\sigma(\omega),T_{\omega_1}(x)).\nonumber
\end{equation}
This system is simpler, and moreover, we can retrieve the digits of our semi-regular continued fraction using the following function,
\begin{equation}
b(\omega,x)=k+\omega_2\nonumber
\end{equation}
where $\omega_1+(-1)^{\omega_1}x\in\left(\frac{1}{k+1},\frac{1}{k}\right]$. Then if we take $\omega'\in\Omega^{\mathbb{N}}$ such that $\omega'_1=0$ and $\omega'_{n+1}=\omega_n$ we have that
\begin{equation}
d_n(\omega,x)=b(R^{n-1}(\omega',x)).\nonumber
\end{equation}

The solution to the problem is then simply to find
\begin{align}
&\lim_{n\to\infty}m(\{x\in[-1,1]:d_n=N\})\nonumber\\
=&\lim_{n\to\infty}\int _{\Omega^{\mathbb{N}}}m(\{x\in[0,1]:b(R^{n-1}(\omega',x))=N\})\,{\rm d}\mathbb{P}(\omega) \nonumber.
\end{align}
The last expression is usually studied via the invariant measure of $R$, or simply the one corresponding the Markov process associated with the random system; i.e., satisfying \eqref{eq:stat}, which we denote by $\mu$. This is where a problem arises. Although we know that such a measure exists and is absolutely continuous with respect to the Lebesgue measure, from \cite{Kalle}, we do not know the formula of its density. We would like to estimate this measure using information that we can write out explicitly. In \cite{Wael} a first order approximation for the invariant density is given. Below we obtain a $k^{th}$ order approximation for any $k\in\mathbb{N}$. We achieve our goal via Theorem \ref{MainProp}. First we show that the required assumptions of Theorem \ref{MainProp} hold for the annealed transfer operator. This is done below.

\subsection{Uniform spectral gap on $C^{l}([0,1])$}

In order to show that the Transfer operator has spectral gap we will use Hennion's theorem \cite{Hennion}. In our case $(C^{l}([0,1]),\norm{\cdot}_{C^{l}})$ is our Banach space.

If $L$ is quasi-compact and if $L$ has $1$ as a unique simple eigenvalue on the unit circle, then it has a spectral gap. For our result we require this spectral gap to be uniform in $\epsilon$, that is there is some $\delta$ such that the difference between the largest two eigenvalues of $L_\epsilon$ is greater than $\delta$ for all $\epsilon$, this can be done using the Keller-Liverani Theorem (Theorem 1 from \cite{Liverani}) as shown below.
\subsubsection{The Keller-Liverani Theorem}\label{SpectGapNote}
From \cite{Liverani}, in order for a uniform spectral gap in $\epsilon$, in addition to the uniform Lasota-Yorke inequality we require another property to hold;
\begin{equation}
|||L_\epsilon-L_0|||\to 0\nonumber
\end{equation}
as $\epsilon\to0$, where
\begin{equation}
|||L|||=\sup_{\norm{f}_{C^l}\leq 1}\norm{Lf}_{C^{l-1}}.\nonumber
\end{equation}
We see this is true for our case since
\begin{align}
&|||L_\epsilon-L_0|||\nonumber\\
=&\sup_{\norm{f}_{C^l}\leq 1}\norm{L_\epsilon f-L_0 f}_{C^{l-1}}\nonumber\\
=&\sup_{\norm{f}_{C^l}\leq 1}\norm{(1-\epsilon) L_0f+\epsilon L_1f -L_0 f}_{C^{l-1}}\nonumber\\
=&\sup_{\norm{f}_{C^l}\leq 1}\norm{\epsilon (L_1-L_0)f }_{C^{l-1}}\nonumber\\
\leq&2\epsilon \cdot M\nonumber
\end{align}
where we have used the fact that $\norm{L_if}_{C^{l-1}}\leq M\norm{f}_{C^{l-1}}$.

We now test the Lasota-yorke inequality.
\subsubsection{The Lasota-Yorke inequality}\label{Spectgap}
We want to show that the Lasota-Yorke inequality holds for $L_\epsilon$ in $(C^{l}([0,1]),\norm{\cdot}_{C^{l}})$ with $\norm{\cdot}_{C^{l-1}}$ as our semi-norm, that is we want to show
\begin{equation}\label{LasotaYorke}
\norm{(L_\epsilon^nf)}_{C^{l}}\leq\theta^n\norm{f}_{C^{l}}+C'_1\norm{f}_{C^{l-1}}
\end{equation}
where $0<\theta<1$ and $C'_1>0$

The transfer operator of the $n^{th}$ iterate of the Gauss-R\'enyi random map is
\begin{equation}\label{eq:nthiterate}
L_\epsilon^nf=\sum_{\omega\in\Omega}P_\omega\sum_{\underline{a}}|(V^\omega_{\underline{a}})'|f\circ V^\omega_{\underline{a}}
\end{equation}
where  $P_\omega$ is the product of $\epsilon^i$ and $(1-\epsilon)^j$ where $i$ is the number of $1$'s in the first $n$ entries of $\omega$, and $j$ is the number of $0$'s in the first $n$ entries, and where $V^\omega_{\underline{a}}$ is the inverse of the branch of the Gauss-R\'enyi map labeled $\underline{a}$, when the map takes the path $\omega$. We will show that this inequality holds for the second iterate, which implies the result for more iterates.
\begin{align}
L_\epsilon  ^2f=&(1-\epsilon)^2\sum_{\underline{a}}|(V^{0,0}_{\underline{a}})'|f\circ V^{0,0}_{\underline{a}}+\epsilon(1-\epsilon)\sum_{\underline{a}}|(V^{1,0}_{\underline{a}})'|f\circ V^{1,0}_{\underline{a}}+\nonumber\\
&(1-\epsilon)\epsilon \sum_{\underline{a}}|(V^{0,1}_{\underline{a}})'|f\circ V^{0,1}_{\underline{a}} +\epsilon^2\sum_{\underline{a}}|(V^{1,1}_{\underline{a}})'|f\circ V^{1,1}_{\underline{a}}\label{eq:L2}
\end{align}
where $V^{p,q}_{\underline{a}}$, $p,q\in\{0,1\}$, is $V^p_{a_1}\circ V^q_{a_2}$ where $V^0_{a}$ is the inverse of the $a^{th}$ branch of the Gauss map and $V^1_{a}$ is the inverse of the $a^{th}$ branch of the R\'enyi map. 

Now we get that
\begin{align}
&(\sum_{\underline{a}}|(V^{0,0}_{\underline{a}})'|f\circ V^{0,0}_{\underline{a}})^{(i)}\nonumber\\
\leq&\norm{f^{(i)}}_{C^0}\sum_{\underline{a}}|((V_{\underline{a}}^{0,0})')^{i}|\nonumber\\
&+\sum_{j=0}^{i-1} \norm{f^{(j)}}_{C^0}  \sum_{\underline{a}}\left[{{i}\choose{j}}|(V^{0,0}_{\underline{a}})^{(i-j+1)}| +\prod_{s=0}^iD_{1s}((V_{\underline{a}}^{0,0})^{(s)})^{D_{2s}}\right]\label{est1}
\end{align}
where $D_{1s}$ and $D_{2s}$ are real non-negative constants depending on $s$, and equally
\begin{align}
&(\sum_{\underline{a}}|(V^{1,1}_{\underline{a}})'|f\circ V^{1,1}_{\underline{a}})^{(i)}\nonumber\\
\leq&\norm{f^{(i)}}_{C^0}\sum_{\underline{a}}|((V_{\underline{a}}^{1,1})')^{i}|\nonumber\\
&+\sum_{j=0}^{i-1} \norm{f^{(j)}}_{C^0} \sum_{\underline{a}}\left[{{i}\choose{j}} |(V^{1,1}_{\underline{a}})^{(i-j+1)}|+\prod_{s=0}^iD_{1s}((V_{\underline{a}}^{1,1})^{(s)})^{D_{2s}}\right].\label{est2}
\end{align}
We have that
\begin{align}
(V^{0,0}_{(n,k)})'(x)=&\frac{1}{(n(k+x)+1)^2}\label{eq:001}\\
(V^{1,1}_{(n,k)})'(x)=&\frac{1}{((n+1)(k+x)-1)^2}.\label{eq:111}
\end{align}

Given that $(V^{0,0}_{\underline{a}})'=-(V^{1,0}_{\underline{a}})'$ and $(V^{1,1}_{\underline{a}})'=-(V^{0,1}_{\underline{a}})'$ these also give us formulas for the $i^{th}$ derivative of $\sum_{\underline{a}}|(V^{1,0}_{\underline{a}})'|f\circ V^{1,0}_{\underline{a}}$ and $\sum_{\underline{a}}|(V^{0,1}_{\underline{a}})'|f\circ V^{0,1}_{\underline{a}} $. This gives us estimates on all the terms of $\norm{(L^2_\epsilon f)^{(i)}}_{C^0}$ in the form that we want, that being in terms of \eqref{est1} and \eqref{est2}.
We can now give an estimate on $\norm{(L^2_\epsilon f)^{(i)}}_{C^0}$. 
\begin{align}
&\norm{(L_\epsilon  ^2f)^{(i)}}_{C^0}\nonumber\\
\leq&(1-\epsilon)^2\norm{(\sum_{\underline{a}}|(V^{0,0}_{\underline{a}})'|f\circ V^{0,0}_{\underline{a}})^{(i)}}_{C^0}\nonumber\\
+&\epsilon(1-\epsilon)\norm{(\sum_{\underline{a}}|(V^{1,0}_{\underline{a}})'|f\circ V^{1,0}_{\underline{a}})^{(i)}}_{C^0}\nonumber\\
+&(1-\epsilon)\epsilon \norm{(\sum_{\underline{a}}|(V^{0,1}_{\underline{a}})'|f\circ V^{0,1}_{\underline{a}})^{(i)}}_{C^0} \nonumber\\
+&\epsilon^2\norm{(\sum_{\underline{a}}|(V^{1,1}_{\underline{a}})'|f\circ V^{1,1}_{\underline{a}})^{(i)}}_{C^0}\nonumber
\end{align}
which we insert our formulas \eqref{est1} and \eqref{est2} into, and reduce, to get
\begin{align}
&\norm{(L_\epsilon  ^2f)^{(i)}}_{C^0}\nonumber\\
\leq&(1-\epsilon)\norm{f^{(i)}}_{C^0}\sum_{\underline{a}}|((V_{\underline{a}}^{0,0})')^{i}|+\epsilon \norm{f^{(i)}}_{C^0}\sum_{\underline{a}}|((V_{\underline{a}}^{1,1})')^{i}|+C_i'\sum_{j=0}^{i-1}\norm{f^{(j)}}_{C^0}\nonumber\\
=&\norm{f^{(i)}}_{C^0}\left[(1-\epsilon)\sum_{\underline{a}}|((V_{\underline{a}}^{0,0})')^{i}|+\epsilon \sum_{\underline{a}}|((V_{\underline{a}}^{1,1})')^{i}|\right]+C_i'\norm{f^{(j)}}_{C^{i-1}}\nonumber\\
=&\theta \norm{f^{(i)}}_{C^0}+C_i'\norm{f^{(j)}}_{C^{i-1}}.\nonumber
\end{align}

where $C_i'$ is the maximum of the sum over $\underline{a}$ in equation \eqref{est1} and the sum over $\underline{a}$ in \eqref{est2}. In order for these estimates to fulfill \eqref{LasotaYorke} we need $C_i'$ finite and $\theta\in(0,1)$. To show that $\theta\in(0,1)$ we will use \eqref{eq:001} and \eqref{eq:111}.

From \eqref{eq:001} we have
\begin{align}
&\sum_{\underline{a}}|((V_{\underline{a}}^{0,0})')^{i}|\nonumber\\
\leq &\sum_{n,k=1}^\infty\left|\frac{1}{(nk+1)^{2i}}\right|\nonumber\\
= &\left(\sum_{n=1}^\infty\left|\frac{1}{n^{2i}}\right|\right)\left(\sum_{k=1}^\infty\left|\frac{1}{(k+\frac{1}{n})^{2i}}\right|\right)\nonumber\\
\leq&\zeta (2i) ^ 2-(1-\frac{1}{2^{2i}})\nonumber
\end{align}
which for integer $i\geq2$ is less than $1$, and $\theta_1=\zeta (2i) ^ 2-(1-\frac{1}{2^{2i}})$ goes to $0$ as $i$ goes to infinity. From \eqref{eq:111} we have
\begin{align}
&\sum_{\underline{a}}|((V_{\underline{a}}^{1,1})')^{i}|\nonumber\\
\leq &\sum_{n,k=1}^\infty\left|\frac{1}{((n+1)k-1)^{2i}}\right|\nonumber\\
=&1+\sum_{n=1}^\infty\sum_{k=2}^\infty\left|\frac{1}{((n+1)k-1)^{2i}}\right|+\sum_{n=2}^\infty\left|\frac{1}{n^{2i}}\right|\nonumber
\end{align}
which can be rewritten as
\begin{align}
&1+\sum_{n=1}^\infty\left|\frac{1}{(n+1)^{2i}}\sum_{k=2}^\infty\frac{1}{(k-\frac{1}{n+1})^{2i}}\right|+\zeta(2i)-1\nonumber\\
\leq&1+\zeta(2i)\sum_{n=1}^\infty\left|\frac{1}{(n+1)^{2i}}\right|+\zeta(2i)-1\nonumber\\
=&1+\zeta(2i)(\zeta(2i)-1)+(\zeta(2i)-1)\nonumber\\
=&1+(\zeta(2i)+1)(\zeta(2i)-1)\nonumber\\
=&1+\zeta(2i)^2-1\nonumber\\
=&\zeta(2i)^2,\nonumber
\end{align}
where $C_i=\zeta(2i)^2$ converges to $1$ as $i$ goes to infinity. These give us
\begin{equation}
\theta\leq(1-\epsilon)\theta_1+\epsilon C_i.\nonumber
\end{equation}
We should note that oll of these constants that give our estimates on $\theta_1$ and $C_i$ are independent of $\epsilon$. We take
\begin{equation}
0\leq\epsilon\leq \frac{1-\theta_1}{C_i-\theta_1}\nonumber
\end{equation}
whicgh gives us $\theta\in(0,1)$. We should note that we now have a restriction on our range of $\epsilon$. Since this relies on $\theta_1$ being less than one this is only shown for $i\geq2$, however for $i=1$ this inequality has been demonstrated in \cite{Wael}.

Now it remains to show that $C_i'$ is finite. To do this we need to show that 
\begin{equation}
\sum_{\underline{a}}\left[{{i}\choose{j}}|(V^{0,0}_{\underline{a}})^{(i-j+1)}|+\prod_{s=0}^iD_{1s}((V_{\underline{a}}^{0,0})^{(s)})^{D_{2s}}\right]<\infty\nonumber
\end{equation}
and
\begin{equation}
\sum_{\underline{a}}\left[{{i}\choose{j}}|(V^{1,1}_{\underline{a}})^{(i-j+1)}|+\prod_{s=0}^iD_{1s}((V_{\underline{a}}^{1,1})^{(s)})^{D_{2s}}\right]<\infty\nonumber
\end{equation}
for any given $i$ and $j$.
From \eqref{eq:001} and \eqref{eq:111} we get the formulas
\begin{align}
(V^{0,0}_{(n,k)})^{(i)}(x)=&\frac{i!n^{(i-1)}}{(n(k+x)+1)^{(i+1)}}\label{eq:00i}\\
(V^{1,1}_{(n,k)})^{(i)}(x)=&\frac{i!(n+1)^{(i-1)}}{((n+1)(k+x)-1)^{(i+1)}}.\label{eq:11i}
\end{align}
which, here $j$ may be an integer between $2$ and $i+1$, gives us
\begin{align}
&\sum_{n,k=1}^\infty(V^{0,0}_{(n,k)})^{(j)}(x)\nonumber\\
=&j!\sum_{n,k=1}^\infty\frac{n^{(j-1)}}{(n(k+x)+1)^{(j+1)}}\nonumber\\
\leq&j!\sum_{n=1}^\infty\sum_{k=1}^\infty\frac{n^{(j-1)}}{(nk+1)^{(j+1)}}\nonumber\\
\leq&j!\left(\sum_{n=1}^\infty\frac{1}{n^2}\right)\left(\sum_{k=1}^\infty\frac{1}{k^{j+1}}\right)\nonumber\\
\leq&j!\frac{\pi^4}{36}\nonumber
\end{align}
the last line of which is gotten by observing that $\sum_{n=1}^\infty\frac{1}{n^2}$ is $\frac{\pi^2}{6}$ and that $\sum_{k=1}^\infty\frac{1}{k^{j+1}}\leq\sum_{n=1}^\infty\frac{1}{n^2}$, and
\begin{align}
&\sum_{n,k=1}^\infty(V^{1,1}_{(n,k)})^{(j)}(x)\nonumber\\
=&j!\sum_{n,k=1}^\infty\frac{(n+1)^{(j-1)}}{((n+1)(k+x)+1)^{(j+1)}}\nonumber\\
\leq&j!\sum_{n=1}^\infty\sum_{k=1}^\infty\frac{(n+1)^{(j-1)}}{((n+1)k+1)^{(j+1)}}\nonumber\\
\leq&j!\left(\sum_{n=1}^\infty\frac{1}{(n+1)^2}\right)\left(\sum_{k=1}^\infty\frac{1}{k^{j+1}}\right)\nonumber\\
\leq&j!\frac{\pi^4}{36}.\nonumber
\end{align}
Now,
\begin{align}
&\sum_{\underline{a}}\prod_{s=0}^iD_{1s}((V_{\underline{a}}^{0,0})^{(s)})^{D_{2s}}\nonumber\\
\leq&\prod_{s=0}^i\sum_{\underline{a}}D_{1s}((V_{\underline{a}}^{0,0})^{(s)})^{D_{2s}}\nonumber\\
\leq&\prod_{s=0}^iD_{1s}(s!)^{D_{2s}}\sum_{n,k=1}^\infty\frac{(n)^{(j-1)D_{2s}}}{(n(k+x)+1)^{(j+1)D_{2s}}}\nonumber\\
\leq&\prod_{s=0}^iD_{1s}(s!)^{D_{2s}}\sum_{n,k=1}^\infty\frac{1}{n^{2D_{2s}}k^{(j+1)D_{2s}}}\nonumber\\
\leq&\prod_{s=0}^i\frac{D_{1s}\pi^4}{36}(s!)^{D_{2s}}\nonumber
\end{align}
which as the finite product of finite numbers is finite. Similarly
\begin{align}
&\sum_{\underline{a}}\prod_{s=0}^iD_{1s}((V_{\underline{a}}^{1,1})^{(s)})^{D_{2s}}\nonumber\\
\leq&\prod_{s=0}^i\sum_{\underline{a}}D_{1s}((V_{\underline{a}}^{1,1})^{(s)})^{D_{2s}}\nonumber\\
\leq&\prod_{s=0}^iD_{1s}(s!)^{D_{2s}}\sum_{n,k=1}^\infty\frac{(n+1)^{(j-1)D_{2s}}}{((n+1)(k+x)+1)^{(j+1)D_{2s}}}\nonumber\\
\leq&\prod_{s=0}^iD_{1s}(s!)^{D_{2s}}\sum_{n,k=1}^\infty\frac{1}{(n+1)^{2D_{2s}}k^{(j+1)D_{2s}}}\nonumber\\
\leq&\prod_{s=0}^i\frac{D_{1s}\pi^4}{36}(s!)^{D_{2s}}\nonumber
\end{align}
is finite. Since these are all the components of $C_i'$ and all of these are finite $C_i'$ must be finite.

This demonstrates the Lasota-Yorke inequality on the transfer operator $L_\epsilon$ for $\epsilon\in(0, \frac{1-\theta_1}{C_1-\theta_1})$. This give us quasi-compactness.

\subsection{Unique and simple eigenvalue}

With quasi-compactness shown the final requirement to show uniqueness of the invariant density in $C^l([0,1])$ is to show that $L_\epsilon$ has a unique and simple eigenvalue on the unit circle. This follows from the proof of Lemma 6.5 in \cite{Wael}.

\subsection{$k$ differentiability of $L_\epsilon h_0$ at $\epsilon=0$}\label{Sect:diffLeps}
Since $P_\omega$ as in \eqref{eq:nthiterate} is the product of $\epsilon^i$ and $(1-\epsilon)^j$ where $i+j=n$, it is differentiable in $\epsilon$ as many times as we need. Now since 
\begin{equation}
L^n_\epsilon h_0=\sum_{\omega\in\Omega^n}P_\omega\sum_{\underline{a}}|(V^\omega_{\underline{a}})'|h_0\circ V^\omega_{\underline{a}}\nonumber
\end{equation}
we have
\begin{equation}
\partial_\epsilon^k L^n_\epsilon h_0=\sum_{\omega\in\Omega^n} [\partial_\epsilon^k P_\omega]\sum_{\underline{a}}|(V^\omega_{\underline{a}})'|h_0\circ V^\omega_{\underline{a}}\nonumber
\end{equation}
so $L^n_\epsilon h_0$ is $k$ differentiable, including at $\epsilon=0$. 

We need to show that this is in $C^l([0,1])$. We can see this by noting that $L_\epsilon:C^l([0,1])\to C^l([0,1])$ means that $L_\epsilon^n h_0\in C^l([0,1])$ if $h_0$ is. The fact that $h_0\in C^l([0,1])$ follows from the fact that $L_\epsilon$ has a spectral gap in $C^l$ as shown previously, so $h_0\in C^{l}([0,1])$. We can see in the formula for $\partial_\epsilon^k L_\epsilon^n h_0$ that regualrity in $x$ is not reduced by differentiating in $\epsilon$, so $\partial_\epsilon^k L_\epsilon^n h_0\in C^l([0,1])$.

\subsection{$k-i$ differentiability of $(I-L_\epsilon)^{-1}G_i$}

Here we will use the fact that if $(I-L_\epsilon)^{-1}G_i$ is well defined and uniformly bounded, then $(I-L_\epsilon)^{-1}G_i=G_i+\sum_{i=1}^\infty L_\epsilon^nG_i$ is well defined. This means it is sufficient to show that $L_\epsilon^n$ is $k-i$ differentiable and such a derivative is in $C^l$.

We may note that 
\begin{equation}
L_\epsilon^nG_i=\sum_{\omega\in\Omega^n}P_\omega\sum_{\underline{a}}|(V^\omega_{\underline{a}})'|[\partial_\epsilon^i L_\epsilon h_0|_{\epsilon=0}]\circ V^\omega_{\underline{a}}.\nonumber
\end{equation}
Like before we have
\begin{equation}
\partial_\epsilon^{k-i} L_\epsilon^nG_i=\sum_{\omega\in\Omega^n} [\partial_\epsilon^{k-i} P_\omega]\sum_{\underline{a}}|(V^\omega_{\underline{a}})'|[\partial_\epsilon^i L_\epsilon h_0|_{\epsilon=0}]\circ V^\omega_{\underline{a}}.\nonumber
\end{equation}
It was shown in the previous section that $\partial_\epsilon^k L_\epsilon h_0$ is $C^l$, and therefore it is at $\epsilon=0$. By the same argument used in the previous section $\partial_\epsilon^{k-i} L_\epsilon^nG_i\in C^l([0,1])$.

\subsection{A formula for the $k^{th}$ order approximation of $h_\epsilon$}
Having demonstrated that all of the assumptions required to use proposition \ref{MainProp} for the Gauss-R\'enyi random map we can estimate the invariant density as follows.
We have from \eqref{Result} that
\begin{equation}
\partial_\epsilon h_\epsilon|_{\epsilon=0}=(I-L_0)^{-1}G_1\nonumber
\end{equation}
and
\begin{equation}
\partial_\epsilon^2 h_\epsilon|_{\epsilon=0}=\sum_{i=1}^2{{2} \choose {i}} H_{i,2-i}=2\cdot \partial_{\epsilon}(I-L_{\epsilon})^{-1}G_1|_{\epsilon=0}+(I-L_0)^{-1}G_2.\nonumber
\end{equation}
For the Gauss-R\'enyi map this gives us
\begin{align}
h_\epsilon=&h_0+\epsilon[ (I-L_0)^{-1}\partial_{\epsilon}L_{\epsilon}h_0|_{\epsilon=0}]\nonumber\\
+&{\epsilon^2}[\partial_{\epsilon}(I-L_{\epsilon})^{-1}\left[\partial_{\epsilon}L_{\epsilon}h_0|_{\epsilon=0}\right]|_{\epsilon=0}]+\frac{\epsilon^2}{2}(I-L_0)^{-1}\left[\partial_{\epsilon}^2L_{\epsilon}h_0|_{\epsilon=0}\right]+o(\epsilon^2)\nonumber
\end{align}
We note however that
\begin{equation}
L_\epsilon h_0 =(1-\epsilon)L_0 h_0+\epsilon L_1 h_0\nonumber
\end{equation}
and so
\begin{equation}
\partial_\epsilon L_\epsilon h_0 =-L_0 h_0+L_1 h_0=L_1 h_0 -h_0=G_1.\nonumber
\end{equation}
Since $\partial_\epsilon L_\epsilon h_0$ is independent of $\epsilon$ we have
\begin{equation}
\partial_\epsilon^k L_\epsilon h_0=0 \nonumber
\end{equation}
for all $k>1$, so $G_i=0$ for $i>1$, so $H_{i,j}=0$ for $i>1$. This means that
\begin{equation}
\partial_\epsilon^n h_\epsilon|_{\epsilon=0}=\sum_{i=1}^n {n \choose i} H_{i,n-i}={n \choose 1}H_{1,n-1}=n\cdot \partial^{n-1}_{\epsilon}(I-L_{\epsilon})^{-1}G_1|_{\epsilon=0}\nonumber
\end{equation}
which gives us that 
\begin{equation}\label{eq:aproxh}
h_\epsilon=h_0+\sum_{n=1}^k \frac{\epsilon^n}{(n-1)!} \partial^{n-1}_{\epsilon}(I-L_{\epsilon})^{-1}[L_1 h_0 -h_0]|_{\epsilon=0}+o(\epsilon^k).
\end{equation}
\subsection{The $n^{th}$ digit of the random continued fraction expansion}\label{subsect:Example}
We can use formula \eqref{eq:aproxh} to find $\lim_{n\to\infty}\int_{\Omega^{\mathbb{N}}}m(\{x\in[-1,1]:d_n=N\})\, {\rm d}\mathbb{P}_\epsilon(\omega)$. As stated in section \ref{sect:randcontfrac} we simply have to find $\int_{\Omega^{\mathbb{N}}}\mu(\{x\in[0,1]:b(\omega',x)=N\})\, {\rm d}\mathbb{P}_\epsilon(\omega)$. In order to do some explicit calculations, since the value of $b(\omega, x)$ depends on the first two symbols in $\omega$, we split $\Omega^{\mathbb N}$ into four mutually disjoint sets.

\begin{equation*}
\begin{split}
&\int_{\Omega^{\mathbb{N}}}\int_{[0,1]}1_{\{N\}}\circ b(\omega',x) h_\epsilon(x)\, {\rm d}x\,{\rm d}\mathbb{P}_\epsilon(\omega)
=\int_{\Omega^{\mathbb{N}}_{0,0}}\int_{(\frac{1}{N+1},\frac{1}{N}]}h_\epsilon(x)\, {\rm d}x\,{\rm d}\mathbb{P}_\epsilon(\omega)\\
&+\int_{\Omega^{\mathbb{N}}_{0,1}}\int_{(\frac{1}{N},\frac{1}{N-1}]}h_\epsilon(x)\, {\rm d}x\,{\rm d}\mathbb{P}_\epsilon(\omega)+\int_{\Omega^{\mathbb{N}}_{1,0}}\int_{[1-\frac{1}{N},1-\frac{1}{N+1})}h_\epsilon(x)\, {\rm d}x\,{\rm d}\mathbb{P}_\epsilon(\omega)\\
&+\int_{\Omega^{\mathbb{N}}_{1,1}}\int_{[1-\frac{1}{N-1},1-\frac{1}{N})}h_\epsilon(x)\, {\rm d}x\,{\rm d}\mathbb{P}_\epsilon(\omega\\
&=\int_{(\frac{1}{N+1},\frac{1}{N}]}h_\epsilon(x)\, {\rm d}x\int_{\Omega^{\mathbb{N}}_{0,0}}1\,{\rm d}\mathbb{P}_\epsilon(\omega)+\int_{(\frac{1}{N},\frac{1}{N-1}]}h_\epsilon(x)\, {\rm d}x\int_{\Omega^{\mathbb{N}}_{0,1}}1\,{\rm d}\mathbb{P}_\epsilon(\omega)\\
&+\int_{[1-\frac{1}{N},1-\frac{1}{N+1})}h_\epsilon(x)\, {\rm d}x\int_{\Omega^{\mathbb{N}}_{1,0}}1\,{\rm d}\mathbb{P}_\epsilon(\omega)
+\int_{[1-\frac{1}{N-1},1-\frac{1}{N})}h_\epsilon(x)\, {\rm d}x\int_{\Omega^{\mathbb{N}}_{1,1}}1\,{\rm d}\mathbb{P}_\epsilon(\omega),
\end{split}
\end{equation*}
where $\Omega^{\mathbb{N}}_{i,j}$ is the subset of $\Omega^{\mathbb{N}}$ such that $\omega_1=i$ and $\omega_2=j$. Since $T_0$ is chosen with probability $0<1-\epsilon<1$ and $T_1$ is chosen with probability $\epsilon$, we get
\begin{align}\label{ExampleExpand}
&\int_{\Omega^{\mathbb{N}}}\int_{[0,1]}1_{\{N\}}\circ b(\omega',x) h_\epsilon(x)\, {\rm d}x\,{\rm d}\mathbb{P}_\epsilon(\omega)
=(1-\epsilon)^2\int_{(\frac{1}{N+1},\frac{1}{N}]}h_\epsilon(x)\, {\rm d}x\\
+&(1-\epsilon)\epsilon\int_{(\frac{1}{N},\frac{1}{N-1}]}h_\epsilon(x)\, {\rm d}x
+\epsilon(1-\epsilon)\int_{[1-\frac{1}{N},1-\frac{1}{N+1})}h_\epsilon(x)\, {\rm d}x\nonumber\\
+&\epsilon^2\int_{[1-\frac{1}{N-1},1-\frac{1}{N})}h_\epsilon(x)\, {\rm d}x.\nonumber
\end{align}
%\begin{remark}
%We can see that if $\epsilon=0$ we have the original deterministic case, giving the Gauss formula \eqref{eq:GaussFormula}. Since
%\begin{equation}
%\int_{\Omega^{\mathbb{N}}_{0,0}}1\,{\rm d}\mathbb{P}_\epsilon(\omega)=1,\qquad\int_{\Omega^{\mathbb{N}}_{0,1}}1\,{\rm d}\mathbb{P}_\epsilon(\omega)=\int_{\Omega^{\mathbb{N}}_{1,0}}1\,{\rm d}\mathbb{P}_\epsilon(\omega)=\int_{\Omega^{\mathbb{N}}_{1,1}}1\,{\rm d}\mathbb{P}_\epsilon(\omega)=0\nonumber
%\end{equation}
%which yields
%\begin{equation*}
%\int_{\Omega^{\mathbb{N}}}\int_{[0,1]}1_{\{N\}}\circ b(\omega',x) h_0(x)\, {\rm d}x\,{\rm d}\mathbb{P}_0(\omega)=\int_{(\frac{1}{N+1},\frac{1}{N}]}h_0(x)\, {\rm d}x=\frac{1}{\log{2}}\log{\frac{1+\frac{1}{N}}{1+\frac{1}{N+1}}}\nonumber
%\end{equation*}
%since $h_0=\frac{1}{\log{2}(1+x)}$.
%\end{remark}
Using our approximation \eqref{eq:aproxh} yields,
\begin{align}
&\int_{\Omega^{\mathbb{N}}}\int_{[0,1]}1_{\{N\}}\circ b(\omega',x) (h_0+\sum_{n=1}^k \frac{\epsilon^n}{(n-1)!} \partial^{n-1}_{\epsilon}(I-L_{\epsilon})^{-1}[L_1 h_0 -h_0]|_{\epsilon=0}+o(\epsilon^k))\, {\rm d}x\,{\rm d}\mathbb{P}(\omega)\nonumber\\
&=(1-\epsilon)^2\int_{(\frac{1}{N+1},\frac{1}{N}]}h_0+\sum_{n=1}^k \frac{\epsilon^n}{(n-1)!} \partial^{n-1}_{\epsilon}(I-L_{\epsilon})^{-1}[L_1 h_0 -h_0]|_{\epsilon=0}+o(\epsilon^k)\, {\rm d}x\nonumber\\
&+(1-\epsilon)\epsilon\int_{(\frac{1}{N},\frac{1}{N-1}]}h_0+\sum_{n=1}^k \frac{\epsilon^n}{(n-1)!} \partial^{n-1}_{\epsilon}(I-L_{\epsilon})^{-1}[L_1 h_0 -h_0]|_{\epsilon=0}+o(\epsilon^k)\, {\rm d}x\nonumber\\
&+\epsilon(1-\epsilon)\int_{[1-\frac{1}{N},1-\frac{1}{N+1})}h_0+\sum_{n=1}^k \frac{\epsilon^n}{(n-1)!} \partial^{n-1}_{\epsilon}(I-L_{\epsilon})^{-1}[L_1 h_0 -h_0]|_{\epsilon=0}+o(\epsilon^k)\, {\rm d}x\nonumber\\
&+\epsilon^2\int_{[1-\frac{1}{N-1},1-\frac{1}{N})}h_0+\sum_{n=1}^k \frac{\epsilon^n}{(n-1)!} \partial^{n-1}_{\epsilon}(I-L_{\epsilon})^{-1}[L_1 h_0 -h_0]|_{\epsilon=0}+o(\epsilon^k)\, {\rm d}x\nonumber\\
&=(1-\epsilon)^2\int_{(\frac{1}{N+1},\frac{1}{N}]}h_0+\sum_{n=1}^k \frac{\epsilon^n}{(n-1)!} \partial^{n-1}_{\epsilon}(I-L_{\epsilon})^{-1}[L_1 h_0 -h_0]|_{\epsilon=0}\, {\rm d}x\nonumber\\
&+(1-\epsilon)\epsilon\int_{(\frac{1}{N},\frac{1}{N-1}]}h_0+\sum_{n=1}^k \frac{\epsilon^n}{(n-1)!} \partial^{n-1}_{\epsilon}(I-L_{\epsilon})^{-1}[L_1 h_0 -h_0]|_{\epsilon=0}\, {\rm d}x\nonumber\\
&+\epsilon(1-\epsilon)\int_{[1-\frac{1}{N},1-\frac{1}{N+1})}h_0+\sum_{n=1}^k \frac{\epsilon^n}{(n-1)!} \partial^{n-1}_{\epsilon}(I-L_{\epsilon})^{-1}[L_1 h_0 -h_0]|_{\epsilon=0}\, {\rm d}x\nonumber\\
&+\epsilon^2\int_{[1-\frac{1}{N-1},1-\frac{1}{N})}h_0+\sum_{n=1}^k \frac{\epsilon^n}{(n-1)!} \partial^{n-1}_{\epsilon}(I-L_{\epsilon})^{-1}[L_1 h_0 -h_0]|_{\epsilon=0}\, {\rm d}x+o(\epsilon^k)\nonumber.\\
\label{al:ResultApprox}
\end{align}
\begin{example}
Here we use these \eqref{al:ResultApprox} to analytically estimate the distribution of $N=5$, using $k=2$, in a random continued fraction with weights $1-\epsilon$ and $\epsilon$. We have
\begin{align}
&\lim_{n\to\infty}\int_{\Omega^{\mathbb{N}}}m(\{x\in[-1,1]:d_n=5\})d\mathbb{P}_\epsilon(\omega)\label{eq:Example1}\\
=&\int_{\Omega^{\mathbb{N}}}\int_{[0,1]}1_{\{5\}}\circ b(\omega',x) h_0\, {\rm d}x\,{\rm d}\mathbb{P}_\epsilon(\omega)\nonumber\\
+&\epsilon\int_{\Omega^{\mathbb{N}}}\int_{[0,1]}1_{\{5\}}\circ b(\omega',x) (I-L_0)^{-1}[L_1 h_0-h_0]\, {\rm d}x\,{\rm d}\mathbb{P}_\epsilon(\omega)\nonumber\\
+&\epsilon^2\int_{\Omega^{\mathbb{N}}}\int_{[0,1]}1_{\{5\}}\circ b(\omega',x)\partial_\epsilon (I-L_\epsilon)^{-1}[L_1 h_0-h_0]|_{\epsilon=0}\, {\rm d}x\,{\rm d}\mathbb{P}_\epsilon(\omega)+o(\epsilon^2)\nonumber.
\end{align}
Using \eqref{al:ResultApprox} we can write as
\begin{align}
&\lim_{n\to\infty}\int_{\Omega^{\mathbb{N}}}m(\{x\in[-1,1]:d_n=5\})d\mathbb{P}_\epsilon(\omega)=(1-\epsilon)^2\int_{(\frac{1}{6},\frac{1}{5}]}h_0(x)\, {\rm d}x\nonumber\\
&=(1-\epsilon)\epsilon\int_{(\frac{1}{5},\frac{1}{4}]}h_0(x)\, {\rm d}x
+\epsilon(1-\epsilon)\int_{[\frac{4}{5},\frac{5}{6})}h_0(x)\, {\rm d}x+\epsilon^2\int_{[\frac{3}{4},\frac{4}{5})}h_0(x)\, {\rm d}x\nonumber\\
&+\epsilon(1-\epsilon)^2\int_{(\frac{1}{6},\frac{1}{5}]}(I-L_0)^{-1}[L_1 h_0-h_0](x)+\epsilon\partial_\epsilon (I-L_\epsilon)^{-1}[L_1 h_0-h_0]|_{\epsilon=0}(x)\, {\rm d}x\nonumber\\
&+(1-\epsilon)\epsilon^2\int_{(\frac{1}{5},\frac{1}{4}]}(I-L_0)^{-1}[L_1 h_0-h_0](x)+\epsilon\partial_\epsilon (I-L_\epsilon)^{-1}[L_1 h_0-h_0]|_{\epsilon=0}(x)\, {\rm d}x\nonumber\\
&+\epsilon^2(1-\epsilon)\int_{[\frac{4}{5},\frac{5}{6})}(I-L_0)^{-1}[L_1 h_0-h_0](x)+\epsilon\partial_\epsilon (I-L_\epsilon)^{-1}[L_1 h_0-h_0]|_{\epsilon=0}(x)\, {\rm d}x\nonumber\\
&+\epsilon^3\int_{[\frac{3}{4},\frac{4}{5})}(I-L_0)^{-1}[L_1 h_0-h_0](x)+\epsilon\partial_\epsilon (I-L_\epsilon)^{-1}[L_1 h_0-h_0]|_{\epsilon=0}(x)\, {\rm d}x+o(\epsilon^2).\nonumber
\end{align}
We compute the first four terms\footnote{All other terms can be computed rigorously using the computer. See \cite{Wael2} for details.} in the above expression and include several terms in $o(\epsilon^2)$ to get:
\begin{align*}
&=(1-\epsilon)^2\frac{\log{\frac{36}{35}}}{\log{2}}
+(1-\epsilon)\epsilon\frac{\log{\frac{25}{24}}}{\log{2}}
+\epsilon(1-\epsilon)\frac{\log{\frac{55}{54}}}{\log{2}}
+\epsilon^2\frac{\log{\frac{36}{35}}}{\log{2}}\\
&+\epsilon(1-\epsilon)^2\int_{(\frac{1}{6},\frac{1}{5}]}(I-L_0)^{-1}[L_1 h_0-h_0](x)+\epsilon\partial_\epsilon (I-L_\epsilon)^{-1}[L_1 h_0-h_0]|_{\epsilon=0}(x)\, {\rm d}x\\
&+(1-\epsilon)\epsilon^2\int_{(\frac{1}{5},\frac{1}{4}]}(I-L_0)^{-1}[L_1 h_0-h_0](x)\, {\rm d}x\\
&+\epsilon^2(1-\epsilon)\int_{[\frac{4}{5},\frac{5}{6})}(I-L_0)^{-1}[L_1 h_0-h_0](x)\, {\rm d}x+o(\epsilon^2).\nonumber
\end{align*} 
%\begin{remark}
%We expect that in the limit as $\epsilon\to 0$ we should get that our approximation goes to the deterministic case. Indeed we can see that as $\epsilon$ goes to $0$ the only non-vanishing term is
%\begin{equation*}
%\lim_{\epsilon\to 0}(1-\epsilon)^2\int_{(\frac{1}{6},\frac{1}{5}]}h_0(x)\, {\rm d}x=\int_{(\frac{1}{6},\frac{1}{5}]}h_0(x)\, {\rm d}x.
%\end{equation*}
%\end{remark}
\end{example}

\end{document}